\documentclass[twoside,11pt]{amsart}
\usepackage{amsmath,amssymb,amsthm}

\begin{document}

\setlength{\textwidth}{145mm} \setlength{\textheight}{203mm}

\frenchspacing
\newcommand{\ie}{i.e. }
\newcommand{\X}{\mathfrak{X}}
\newcommand{\W}{\mathcal{W}}
\newcommand{\F}{\mathcal{F}}
\newcommand{\T}{\mathcal{T}}
\newcommand{\U}{\mathcal{U}}
\newcommand{\M}{(M,\f,\xi,\eta,g)}
\newcommand{\Lf}{(G,\f,\xi,\eta,g)}
\newcommand{\R}{\mathbb{R}}
\newcommand{\s}{\mathfrak{S}}
\newcommand{\n}{\nabla}
\newcommand{\f}{\varphi}
\newcommand{\D}{{\rm d}}
\newcommand{\al}{\alpha}
\newcommand{\bt}{\beta}
\newcommand{\gm}{\gamma}
\newcommand{\lm}{\lambda}
\newcommand{\ta}{\theta}
\newcommand{\om}{\omega}
\newcommand{\ea}{\varepsilon_\alpha}
\newcommand{\eb}{\varepsilon_\beta}
\newcommand{\eg}{\varepsilon_\gamma}
\newcommand{\sx}{\mathop{\mathfrak{S}}\limits_{x,y,z}}
\newcommand{\norm}[1]{\left\Vert#1\right\Vert ^2}
\newcommand{\nf}{\norm{\n \f}}
\newcommand{\Span}{\mathrm{span}}

\newcommand{\thmref}[1]{The\-o\-rem~\ref{#1}}
\newcommand{\propref}[1]{Pro\-po\-si\-ti\-on~\ref{#1}}
\newcommand{\secref}[1]{\S\ref{#1}}
\newcommand{\lemref}[1]{Lem\-ma~\ref{#1}}
\newcommand{\dfnref}[1]{De\-fi\-ni\-ti\-on~\ref{#1}}
\newcommand{\corref}[1]{Corollary~\ref{#1}}



\numberwithin{equation}{section}
\newtheorem{thm}{Theorem}[section]
\newtheorem{lem}[thm]{Lemma}
\newtheorem{prop}[thm]{Proposition}
\newtheorem{cor}[thm]{Corollary}

\theoremstyle{definition}
\newtheorem{defn}{Definition}[section]

\hyphenation{Her-mi-ti-an ma-ni-fold ah-ler-ian}

\title[Curvature properties on some classes] {Curvature properties on some classes of almost
contact manifolds with B-metric}

\author{Mancho Manev}

\maketitle

{\small
\textbf{Abstract}

Almost contact manifolds with B-metric are considered. Of special
interest are the so-called vertical classes of the almost contact
B-metric manifolds. Curvature properties of these manifolds are
studied. An example of 5-dimensional manifolds is constructed and
characterized.

\textbf{Key words:} almost contact manifold, B-metric, indefinite
metric, Lie group.

\textbf{2010 Mathematics Subject Classification:} 53C15, 53D15,
53C50. }



\section{Introduction}\label{sec1}

In this work\footnote{Partially supported by project NI11-FMI-004
of the Scientific Research Fund, Paisii Hilendarski University of
Plovdiv, Bulgaria} we continue the investigations on a manifold
$M$ with an almost contact structure $(\f,\xi,\eta)$ which is
equipped with a B-metric $g$, i.e. a pseudo-Riemannian metric of
signature $(n,n+1)$ with the opposite compatibility of the metric
with the structure in comparison with the known almost contact
metric structure. Moreover, the B-metric is an odd-dimensional
analogue of the Norden metric on almost complex manifolds.

Recently, manifolds with neutral metrics and various tensor
structures have been object of interest in theoretical physics.

The classes of the almost contact B-metric manifolds from the
so-called vertical component is an object of special interest in
this paper.
The goal of the present work is the investigation of some problems
of the geometry of such manifolds.

The paper is organized as follows. In Sec.~\ref{sec2} we give some
necessary facts about the considered manifolds.
In Sec.~\ref{sec3} we obtain some curvature properties on almost
contact B-metric manifolds in these classes.
In Sec.~\ref{sec4} we consider and characterize a family of
5-dimensional Lie groups as almost contact B-metric manifolds from
a basic class.

\section{Almost contact manifolds with B-metric}\label{sec2}

Let $(M,\f,\xi,\eta,g)$ be an almost contact manifold with
B-metric or an \emph{almost contact B-metric manifold}, i.e. $M$
is a $(2n+1)$-dimensional differentiable manifold with an almost
contact structure $(\f,\xi,\eta)$ consists of an endomorphism $\f$
of the tangent bundle, a vector field $\xi$, its dual 1-form
$\eta$ as well as $M$ is equipped with a pseudo-Riemannian metric
$g$  of signature $(n,n+1)$, such that the following algebraic
relations are satisfied
\begin{equation}\label{str}
\begin{array}{c}
\f\xi = 0,\quad \f^2 = -Id + \eta \otimes \xi,\quad
\eta\circ\f=0,\quad \eta(\xi)=1,\\[4pt]
g(\f x, \f y ) = - g(x, y ) + \eta(x)\eta(y)
\end{array}
\end{equation}
for arbitrary $x$, $y$ of the algebra $\X(M)$ on the smooth vector
fields on $M$.

The associated B-metric $\tilde{g}$ of $g$ on $M$ is defined by
$\tilde{g}(x,y)=g(x,\f y)+\eta(x)\eta(y)$. Both metrics are
necessarily of signature $(n,n+1)$. The manifold
$(M,\f,\xi,\eta,\tilde{g})$ is an almost contact B-metric
manifold, too.

Further, $x$, $y$, $z$, $w$ will stand for arbitrary elements of
$\X(M)$ unless other is specialized.

The structural group of $\M$ is $G\times I$, where $I$ is the
identity on $\Span(\xi)$ and $G=\mathcal{GL}(n;\mathbb{C})\cap
\mathcal{O}(n,n)$. More precisely, $G$ consists of real square
matrices of order $2n$ of the following type
\[
\left(%
\begin{array}{r|r}
  A & B \\ \hline
  -B & A \\
\end{array}%
\right),
\qquad %
\begin{array}{l}
  AA^T-BB^T=I_n,\\[4pt]%
  AB^T+BA^T=O_n,
\end{array}%
\quad A, B\in \mathcal{GL}(n;\mathbb{R}),
\]
where $I_n$ and $O_n$ are the unit matrix and the zero matrix of
size $n$, respectively.

A classification of the almost contact B-metric manifolds is given
in \cite{GaMiGr}. This classification is made with respect to the
(0,3)-tensor $F$ defined by
\begin{equation}\label{1.2}
F(x,y,z)=g\bigl( \left( \nabla_x \f \right)y,z\bigr),
\end{equation}
where $\nabla$ is the Levi-Civita connection of $g$. The tensor
$F$ has the properties
\begin{equation*}\label{1.3}
F(x,y,z)=F(x,z,y)=F(x,\f y,\f
z)+\eta(y)F(x,\xi,z)+\eta(z)F(x,y,\xi).
\end{equation*}
Let us recall the following general relations \cite{ManGri1}:
\begin{equation}\label{eta-xi}
    F(x,\f
    y,\xi)=\left(\n_x\eta\right)y=g\left(\n_x\xi,y\right),\qquad
    \eta\left(\n_x \xi\right)=0.
\end{equation}

The components of the inverse matrix of $g$ are denoted by
$g^{ij}$ with respect to the basis $\{e_i;\xi\}$ of the tangent
space $T_pM$ of $M$ at an arbitrary point $p\in M$. The following
basic 1-forms are associated with $F$ for any $z\in T_pM$:
\begin{equation}\label{lee}
    \theta(z)=g^{ij}F(e_i,e_j,z),\quad
    \theta^*(z)=g^{ij}F(e_i,\f e_j,z),\quad
    \omega(z)=F(\xi,\xi,z).
\end{equation}

The basic classes in the mentioned classification are $\F_1$,
$\F_2$, $\dots$, $\F_{11}$. Their intersection is the class $\F_0$
determined by the condition $F=0$.

In \cite{Man}, it is proved that the class
$\U=\F_4\oplus\F_5\oplus\F_6\oplus\F_7\oplus\F_8\oplus\F_9$ is
defined by the conditions
\begin{equation}\label{F4-9}
    F(x,y,z)=\eta(y)F(x,z,\xi)+\eta(z)F(x,y,\xi),\qquad
    F(\xi,y,z)=0.
\end{equation}

It is known \cite{Sa}, an almost contact manifold
$(M,\f,\xi,\eta)$ is called \emph{normal} if the corresponding
almost complex structure $J$ on the even-dimensional manifold
$M\times\R$ is integrable, i.e. the Nijenhuis torsion
$[J,J](x,y)=-[x,y]+[Jx,Jy]-J[Jx,y]-J[x,Jy]$ is identically zero.
According to \cite{Blair}, the necessary and sufficient condition
$(M,\f,\xi,\eta)$ to be normal is the annulment of the Nijenhuis
tensor $N$ defined by $N(x, y ) := [\f, \f](x, y ) + \D\eta(x, y
)\xi$, where $[\f,\f](x,y)=\f^2[x,y]+[\f x,\f y]-\f[\f x,y]-\f
[x,\f y]$. It is known from \cite{Man} that the class of the
normal contact B-metric manifolds is
$\F_1\oplus\F_2\oplus\F_4\oplus\F_5\oplus\F_6$. Moreover, $\eta$
is closed there, i.e. $\D\eta=0$. Therefore, also we have
$[\f,\f]=0$ for the normal contact B-metric manifolds.

In \cite{Man}, there are considered the linear projectors $h$ and
$v$ over $T_pM$ which split (orthogonally and invariantly with
respect to the structural group) any vector $x$ to a horizontal
component $h(x)=-\f^2x$ and a vertical component
$v(x)=\eta(x)\xi$. The decomposition $T_pM=h(T_pM)\oplus v(T_pM)$
generates the corresponding distribution of basic tensors $F$,
which gives the horizontal component $\F_1\oplus\F_2$ and the
vertical component $\F_4\oplus\F_5\oplus\F_6$ of the class of the
normal contact B-metric manifolds.


By the additional conditions
\begin{equation}\label{F4-6}
F(x,y,\xi)=F(y,x,\xi)=-F(\f x,\f y,\xi),
\end{equation}
it is determined the subclass
$\U_1=\F_4\oplus\F_5\oplus\F_6\subset\U$, which is the class of the normal
contact manifolds from the vertical classes.
This  subclass is also characterized by \eqref{F4-9} and $N=0$. The
classes $\F_4$, $\F_5$ and $\F_6$ are determined in $\U_1$ by
the conditions $\theta^*=0$, $\theta=0$ and $\theta=\theta^*=0$,
respectively.
Another direct sum of basic classes considered below is
$\U_2=\F_4\oplus\F_5\oplus\F_6\oplus\F_7\subset\U$, which is determined by
\eqref{F4-9} and
\begin{equation}\label{F4-7}
F(x,y,\xi)=-F(\f x,\f y,\xi).
\end{equation}

By analogy with the square norm of $\nabla J$ for an almost
complex structure $J$, we define the \emph{square norm of $\nabla
\f$} by \cite{Man31}
\begin{equation}\label{snf}
    \norm{\nabla \f}=g^{ij}g^{ks}
    g\bigl(\left(\nabla_{e_i} \f\right)e_k,\left(\nabla_{e_j}
    \f\right)e_s\bigr).
\end{equation}
It is clear that $\norm{\nabla \f}=0$ is valid if $\M$ is a
$\F_0$-manifold, but the inverse implication is not always true.
An almost contact B-metric manifold having a zero square norm of
$\n\f$ is called an \emph{isotropic-$\F_0$-manifold} \cite{Man31}.
Ana\-logously, we consider the square norm of $\n\eta$ and $\n\xi$
as follows:
\begin{equation*}\label{snxi}
    \norm{\nabla \eta}=g^{ij}g^{ks}
    \left(\nabla_{e_i}\eta\right)e_k\left(\nabla_{e_j}\eta\right)e_s, \quad
    \norm{\nabla \xi}=g^{ij}
    g\bigl(\nabla_{e_i} \xi,\nabla_{e_j}
    \xi\bigr).
\end{equation*}
Obviously, the relation $\norm{\nabla \eta}=\norm{\nabla \xi}$ is
valid, because of \eqref{eta-xi}.

\begin{prop}\label{prop-sn-4-9}
On any almost contact B-metric manifold in $\U$ the following
relation holds
$
\norm{\nabla \f}=-2\norm{\nabla \eta}=-2\norm{\nabla \xi}.
$
\end{prop}
\begin{proof}
It follows immediately from \eqref{F4-9}.
\end{proof}

\section{Curvature properties on the class $\U$}\label{sec3}

Let $R=\left[\n,\n\right]-\n_{[\ ,\ ]}$ be the curvature
(1,3)-tensor of $\nabla$. We denote the curvature $(0,4)$-tensor
by the same letter: $R(x,y,z,w)$ $=g(R(x,y)z,w)$. It has the
properties:
\begin{equation}\label{R}
    R(x,y,z,w)=-R(y,x,z,w)=-R(x,y,w,z),\;
    \sx R(x,y,z,w)=0,
\end{equation}
where $\sx$ denotes the cyclic sum by $x,y,z$.

The property $R(x,y,\f z,\f w)=-R(x,y,z,w)$ is valid for $R$ on a
$\F_0$-mani\-fold, because of $\n\f=0$ there.

The Ricci tensor $\rho$ and the scalar curvature $\tau$ for $R$ as
well as their associated quantities are defined respectively by
\begin{gather}
    \rho(y,z)=g^{ij}R(e_i,y,z,e_j),\qquad
    \tau=g^{ij}\rho(e_i,e_j),\label{rho}\\[4pt]
    \rho^*(y,z)=g^{ij}R(e_i,y,z,\f e_j),\qquad
    \tau^*=g^{ij}\rho^*(e_i,e_j).\label{rho*}
\end{gather}

Further we use the notation $g\odot h$ for the Kulkarni-Nomizu
product of two (0,2)-tensors, \ie $ \left(g\odot
h\right)(x,y,z,w)=g(x,z)h(y,w)-g(y,z)h(x,w)
+g(y,w)h(x,z)-g(x,w)h(y,z). $
Obviously, $g\odot h$ has the properties of $R$ in \eqref{R} when $g$ and $h$ are symmetric.

\begin{thm}\label{thm-R-F4-9}
On any almost contact B-metric manifold in $\U$ the following
relation holds
\begin{equation}\label{Rfih}
\begin{split}
2R(x,y,\f^2 z,\f^2 w)&+2R(x,y,\f z,\f w)\\[4pt]
&=\left(h\odot h\right)(x,y,z,w)+\left(h\odot h\right)(x,y,\f z,\f
w),
\end{split}
\end{equation}
where $h(x,y)=\left(\n_x\eta\right)y$.
\end{thm}
\begin{proof}
By direct computations from \eqref{F4-9} and the Ricci
identity
$
\left(\n_x\n_y \f\right)z-\left(\n_y\n_x \f\right)z=R(x,y)\f z-\f
R(x,y)z$ follows \eqref{Rfih}.
\end{proof}

\begin{cor}\label{cor-U1}
On any almost contact B-metric manifold in $\U_1$ we have
\[ \sx\bigl\{R(x,y,\f^2 z,\f^2
w)+R(x,y,\f z,\f w)\bigr\}=0.
\]
\end{cor}
\begin{proof}
When a manifold belongs to  $\U_1$ then the corresponding
(0,2)-tensor $h$ is symmetric and hybrid with respect to $\f$.
Then the cyclic sum by $x,y,z$ of the right-hand side of
\eqref{Rfih} vanishes, which accomplish the proof.
\end{proof}

\begin{cor}\label{cor-U2}
On any almost contact B-metric manifold in $\U_2$ we have
\begin{equation*}\label{Rfi}
\begin{split}
R(\f x,\f y,\f^2 z,\f^2 w)&+R(\f x,\f y,\f z,\f w)\\[4pt]
&-R(x,y,\f^2 z,\f^2 w)-R(x,y,\f z,\f w)=0.
\end{split}
\end{equation*}
\end{cor}
\begin{proof}
Since $g(x,\f y)=g(\f x,y)$ holds and \eqref{F4-7} is valid on a
manifold in $\U_2$, then using
\eqref{eta-xi} we obtain
$
    \n_{\f x}\xi=\f\n_x\xi.
$
The last equality implies that $h(x,y)$ is a hybrid with respect
to $\f$ and then the right-hand side of \eqref{Rfih} is equal to
the first line of \eqref{Rfi}.
Thus, the statement follows.
\end{proof}

\section{A Lie group as a 5-dimensional $\F_6$-manifold}\label{sec4}

Let $G$ be a 5-dimensional real connected Lie group and let
$\mathfrak{g}$ be its Lie algebra. If $\left\{e_i\right\}$ is a
global basis of left-invariant vector fields of $G$, we define an
almost contact structure $(\f,\xi,\eta)$ and a B-metric $g$ on $G$
as follows:
\begin{equation}\label{f}
\begin{array}{c}
\f e_1 = e_3,\quad \f e_2 = e_4,\quad \f e_3 =-e_1,\quad \f e_4 =
-e_2,\quad \f e_5 =0;\\[4pt]
\xi=e_5;\qquad \eta(e_i)=0\; (i=1,2,3,4),\quad \eta(e_5)=1;
\end{array}
\end{equation}
\begin{equation}\label{g}
\begin{array}{c}
g(e_1,e_1)=g(e_2,e_2)=-g(e_3,e_3)=-g(e_4,e_4)=g(e_5,e_5)=1,
\\[4pt]
g(e_i,e_j)=0,\; i\neq j,\quad  i,j\in\{1,2,3,4,5\}.
\end{array}
\end{equation}

We verify that \eqref{str} are satisfied for $\{e_i\}$. Thus, we
establish that the induced 5-dimensional manifold $\Lf$ is an
almost contact B-metric manifold.

Using the known property of the Levi-Civita connection $\n$ of  $g$
\begin{equation}\label{LC}
2g(\n_{e_i}
e_j,e_k)=g([e_i,e_j],e_k)+g([e_k,e_i],e_j)+g([e_k,e_j],e_i),
\end{equation}
we obtain the form of the basic tensor $F$ as follows
\begin{equation}\label{F}
\begin{split}
2F(e_i,e_j,e_k)= %
g\bigl([e_i,\f e_j]-\f[e_i,e_j],e_k\bigr)%
+&g\bigl([e_i,\f e_k]-\f[e_i,e_k],e_j\bigr)\\[4pt]%
+&g\bigl([e_k,\f e_j]-[\f e_k,e_j],e_i\bigr).%
\end{split}
\end{equation}
On any manifold in $\U_2$ the properties $F(e_i,e_j,\xi)=-F(\f
e_i,\f e_j,\xi)$ and $F(\xi,e_i,e_j)=0$ are valid. According to
\eqref{1.2} and \eqref{eta-xi}, the last two properties are
equivalent to $\n_{\f e_i}\xi=\f\n_{e_i}\xi$ and $\n_\xi \f
e_i=\f\n_\xi e_i$, respectively. These equalities imply the
following property on a manifold in $\U_2$
\begin{equation}\label{F6-1}
    [\f e_i,\xi]=\f [e_i,\xi].
\end{equation}
Using \eqref{F4-6} and  \eqref{F}, the following property is valid
for a manifold in $\U_1$
\begin{equation}\label{F6-2}
    \eta\bigl([e_i,e_j]\bigr)=0.
\end{equation}
To be the considered manifold in $\F_6$ we set $\theta(\xi)=\theta^*(\xi)=0$. Then, using \eqref{lee}
and the equality
$
2F(e_i,e_j,\xi)= %
-g\bigl([e_i,\xi],\f e_j\bigr)%
-g\bigl([e_j,\xi],\f e_i\bigr),
$
which is a corollary of \eqref{F}, we obtain conditions
\begin{equation}\label{leeF6}
    \theta(\xi)=-g^{ij}g\bigl([e_i,\xi],\f e_j\bigr)=0,\quad %
\theta^*(\xi)=g^{ij}g\bigl([e_i,\xi],e_j\bigr)=0.
\end{equation}

Now, let us consider the Lie algebra $\mathfrak{g}$ on $G$,
determined by the following non-zero commutators satisfying
conditions \eqref{F6-1} and \eqref{F6-2}:
\begin{equation*}\label{com1}
\begin{array}{l}
[e_1,\xi]=\lm_1e_1+\lm_2e_2+\lm_3e_3+\lm_4e_4,\\[4pt]
[e_2,\xi]=\mu_1e_1+\mu_2e_2+\mu_3e_3+\mu_4e_4, \\[4pt]
[e_3,\xi]=-\lm_3e_1-\lm_4e_2+\lm_1e_3+\lm_2e_4,\\[4pt]
[e_4,\xi]=-\mu_3e_1-\mu_4e_2+\mu_1e_3+\mu_2e_4,
\end{array}
\end{equation*}
where $\lm_i, \mu_i \in \R$ $(i=1,2,3,4)$
. 

We compute $\theta(\xi)=2(\lm_1+\mu_2)$ and
$\theta^*(\xi)=2(\lm_3+\mu_4)$ using \eqref{leeF6} and then we
determine $\mu_2=-\lm_1$ and $\mu_4=-\lm_3$. Therefore,  we obtain
\begin{equation}\label{com2}
\begin{split}
&[e_1,\xi]=\lm_1e_1+\lm_2e_2+\lm_3e_3+\lm_4e_4,
\\[4pt]
&[e_2,\xi]=\mu_1e_1-\lm_1e_2+\mu_3e_3-\lm_3e_4,
\\[4pt]
&[e_3,\xi]=-\lm_3e_1-\lm_4e_2+\lm_1e_3+\lm_2e_4,
\\[4pt]
&[e_4,\xi]=-\mu_3e_1+\lm_3e_2+\mu_1e_3-\lm_1e_4.
\end{split}
\end{equation}
It is easy to obtain that the other commutators $[e_i,e_j]$
$(i,j\in\{1,2,3,4\})$ are zero.

We verify immediately that the Jacobi identity of $[e_i,e_j]$ for
all  $i,j\in\{1,2,3,4,5\}$ is satisfied.

Using \eqref{F} and \eqref{com2}, we obtain that the non-zero components of $F$ are:
\begin{equation}\label{Fijk}
\begin{array}{c}
    F_{115}=-F_{225}=-F_{335}=F_{445}=\lm_3,\quad
F_{135}=-F_{245}=\lm_1,\\[4pt]
F_{145}=F_{235}=\frac{1}{2}\left(\lm_2+\mu_1\right),\quad
F_{125}=-F_{345}=\frac{1}{2}\left(\lm_4+\mu_3\right),
\end{array}
\end{equation}
where $F_{ij5}=F(e_i,e_j,\xi)$, $i,j\in\{1,2,3,4\}$.

\begin{thm}
Let $(G,\f,\xi,\eta,g)$ be the almost contact B-metric manifold,
determined by \eqref{f}, \eqref{g} and \eqref{com2}. Then it
belongs to the class $\F_6$.
\end{thm}

\begin{proof}
We check directly that the components in \eqref{Fijk} satisfy
conditions \eqref{F4-9}, \eqref{F4-6} and $\ta=\ta^*=0$ for
$\F_6$. Therefore we establish that the corresponding manifold
$\Lf$ belongs to  $\F_6$. It is a $\F_0$-manifold if
and only if the condition
$\lm_1=\lm_2+\mu_1=\lm_3=\lm_4+\mu_3= 0$ holds.
\end{proof}

Using \eqref{LC} and \eqref{com2}, we obtain the components of $\n$:
\begin{equation}\label{nabli}
\begin{array}{l}
\n_{e_1}e_1=-\n_{e_2}e_2=-\n_{e_3}e_3=\n_{e_4}e_4=-\lm_1\xi,\quad \n_{\xi}\xi=0, \\[4pt]
\n_{e_1}e_2=\n_{e_2}e_1=-\n_{e_3}e_4=-\n_{e_4}e_3=-\frac{1}{2}\left(\lm_2+\mu_1\right), \\[4pt]
\n_{e_1}e_3=-\n_{e_2}e_4=\n_{e_3}e_1=-\n_{e_4}e_2=\lm_3\xi, \\[4pt]
\n_{e_1}e_4=\n_{e_2}e_3=\n_{e_3}e_2=\n_{e_4}e_1=\frac{1}{2}\left(\lm_4+\mu_3\right), \\[4pt]
\n_{e_1}\xi=\lm_1e_1+\frac{1}{2}\left(\lm_2+\mu_1\right)e_2+\lm_3e_3+\frac{1}{2}\left(\lm_4+\mu_3\right)e_4, \\[4pt]
\n_{e_2}\xi=\frac{1}{2}\left(\lm_2+\mu_1\right)e_1-\lm_1e_2+\frac{1}{2}\left(\lm_4+\mu_3\right)e_3-\lm_3e_4, \\[4pt]
\n_{e_3}\xi=-\lm_3e_1-\frac{1}{2}\left(\lm_4+\mu_3\right)e_2+\lm_1e_3+\frac{1}{2}\left(\lm_2+\mu_1\right)e_4, \\[4pt]
\n_{e_4}\xi=-\frac{1}{2}\left(\lm_4+\mu_3\right)e_1+\lm_3e_2+\frac{1}{2}\left(\lm_2+\mu_1\right)e_3-\lm_1e_4, \\[4pt]
\n_{\xi}e_1=-\frac{1}{2}\left(\lm_2-\mu_1\right)e_2-\frac{1}{2}\left(\lm_4-\mu_3\right)e_4, \\[4pt]
\n_{\xi}e_2=\frac{1}{2}\left(\lm_2-\mu_1\right)e_1+\frac{1}{2}\left(\lm_4-\mu_3\right)e_3,\\[4pt]
\n_{\xi}e_3=\frac{1}{2}\left(\lm_4-\mu_3\right)e_2-\frac{1}{2}\left(\lm_2-\mu_1\right)e_4,\\[4pt]
\n_{\xi}e_4=-\frac{1}{2}\left(\lm_4-\mu_3\right)e_1+\frac{1}{2}\left(\lm_2-\mu_1\right)e_3.
\end{array}
\end{equation}


Let us denote $R_{ijkl}=R(e_{i},e_{j},e_{k},e_{l})$ and $\rho_{jk}=\rho(e_{j},e_{k})$.
According to (\ref{g}) and (\ref{nabli}), we obtain
\begin{equation}\label{Rijkl}
\begin{split}
R_{5115}&=-R_{5335}=\rho_{11}=-\rho_{33}\\[4pt]
&=-\lm_1^2+\lm_3^2-\frac{3}{4}\left(\lm_2^2-\lm_4^2\right)
+\frac{1}{4}\left(\mu_1^2-\mu_3^2\right)-\frac{1}{2}\left(\lm_2\mu_1-\lm_4\mu_3\right),
\\[4pt]
R_{5225}&=-R_{5445}=\rho_{22}=-\rho_{44}\\[4pt]
&=-\lm_1^2+\lm_3^2+\frac{1}{4}\left(\lm_2^2-\lm_4^2\right)
-\frac{3}{4}\left(\mu_1^2-\mu_3^2\right)-\frac{1}{2}\left(\lm_2\mu_1-\lm_4\mu_3\right),
\\[4pt]
R_{5125}&=-R_{5345}=\rho_{12}=-\rho_{34}=\lm_1\left(\lm_2-\mu_1\right)-\lm_3\left(\lm_4-\mu_3\right),\\[4pt]
R_{5145}&=R_{5235}=\rho_{14}=\rho_{23}=-\lm_1\left(\lm_4-\mu_3\right)-\lm_3\left(\lm_2-\mu_1\right),\\[4pt]
R_{5135}&=\rho_{13}=2\lm_1\lm_3+\frac{3}{2}\lm_2\lm_4-\frac{1}{2}\mu_1\mu_3+\frac{1}{2}\left(\lm_2\mu_3+\lm_4\mu_1\right),\\[4pt]
R_{5245}&=\rho_{24}=2\lm_1\lm_3-\frac{1}{2}\lm_2\lm_4+\frac{3}{2}\mu_1\mu_3+\frac{1}{2}\left(\lm_2\mu_3+\lm_4\mu_1\right),\\[4pt]
2\rho_{55}&=-8\left(\lm_1^2-\lm_3^2\right)
-2\left(\lm_2+\mu_1\right)^2 +2\left(\lm_4+\mu_3\right)^2=\tau.
\end{split}
\end{equation}
The rest of the non-zero components of  $R$ and  $\rho$ are determined by \eqref{Rijkl} and the properties
$R_{ijkl}=R_{klij}$, $R_{ijkl}=-R_{jikl}=-R_{ijlk}$  and $\rho_{jk}=\rho_{kj}$. Moreover, the
associated scalar curvature is
\begin{equation}\label{tau*}
\tau^*=8\lm_1\lm_3+2\left(\lm_2+\mu_1\right)\left(\lm_4+\mu_3\right).
\end{equation}

It is easy to check that the components of $R$ in \eqref{Rijkl}
satisfy the equalities in \thmref{thm-R-F4-9}, \corref{cor-U1} and
\corref{cor-U2}.

\begin{thm}\label{thm-isoF0}
Let $\Lf$ is the $\F_6$-manifold determined by \eqref{f},
\eqref{g} and \eqref{com2}. Then the following properties are
equivalent:
\begin{enumerate}\renewcommand{\labelenumi}{(\roman{enumi})}
    \item $\Lf$ is an isotropic-$\F_0$-manifold;
    \item $\Lf$ is scalar flat;
    \item $4\left(\lm_3^2-\lm_1^2\right)
-\left(\lm_2+\mu_1\right)^2+\left(\lm_4+\mu_3\right)^2=0$.
\end{enumerate}
\end{thm}
\begin{proof}
By virtue of \eqref{snf} and  \eqref{Fijk} we get
$\nf=4\left(\lm_3^2-\lm_1^2\right)-\left(\lm_2+\mu_1\right)^2
+\left(\lm_4+\mu_3\right)^2$. Having in mind
\propref{prop-sn-4-9}, the latter equality and the last line of
\eqref{Rijkl}, we have
\[
\begin{split}
\tau&=2\nf=-4\norm{\nabla \eta}=-4\norm{\nabla \xi}\\[4pt]
&=8\left(\lm_3^2-\lm_1^2\right)
-2\left(\lm_2+\mu_1\right)^2+2\left(\lm_4+\mu_3\right)^2.
\end{split}
\]
Then the statement follows immediately.
\end{proof}

Let us remark, there exists such a $\F_6$-manifold which is a
scalar flat isotropic-$\F_0$-manifold but it is not a
$\F_0$-manifold. For example, such a case is when $\lm_1=\lm_3\neq
0$, $\mu_1=-\lm_2$ and $\mu_3=-\lm_4$.

According to \eqref{g}, \eqref{rho}, \eqref{Rijkl} and
\eqref{tau*}, we obtain
\begin{prop}\label{prop-almEinst}
The $\F_6$-manifold $\Lf$, determined by \eqref{f}, \eqref{g} and
\eqref{com2}, is almost Einsteinian, i.e. $\rho_{ij}=\nu
g_{ij}+\tilde{\nu}\tilde{g}_{ij}$, if and only if the following
conditions are valid
\[
\mu_1=\lm_2,\qquad \mu_3=\lm_4, \qquad
3\left(\lm_1^2+\lm_2^2-\lm_3^2-\lm_4^2\right)-2\left(\lm_1\lm_3+\lm_2\lm_4\right)=0.
\]
In this case we have
$$
\tau=8\nu=-8\left(\lm_1^2+\lm_2^2-\lm_3^2-\lm_4^2\right),\qquad
\tau^*=-4\tilde{\nu}=8\left(\lm_1\lm_3+\lm_2\lm_4\right),
$$ i.e.
$3\tau+2\tau^*=0$ and
$\rho_{ij}=\frac{\tau}{8}\left(g_{ij}+3\tilde{g}_{ij}\right)$.
 \qed
\end{prop}

Bearing in mind \thmref{thm-isoF0} and \propref{prop-almEinst}, we
have immediately
\begin{cor}
The $\F_6$-manifold $\Lf$, determined by \eqref{f}, \eqref{g} and
\eqref{com2}, is a scalar flat almost Einsteinian
isotropic-$\F_0$-manifold, if and only if the following condition
is valid
$
\lm_1^2+\lm_2^2-\lm_3^2-\lm_4^2=0.
$
 \qed
\end{cor}

\bigskip

\small{ \noindent
\textsl{Mancho Manev\\
Department of Geometry\\
Faculty of Mathematics and Informatics\\
University of Plovdiv\\
236 Bulgaria Blvd\\
4003 Plovdiv, Bulgaria}
\\
\texttt{e-mail: mmanev@uni-plovdiv.bg} }

\end{document}